%% file: RevisionAfterAIforArXiv.tex
\documentclass[reqno]{amsart}

\input{Preamble}

\newcommand{\tcd}{tree-cut de\-com\-po\-si\-tion}

\newcommand{\alphathin}{$\alpha$-thin}
\newcommand{\almostalphathin}{almost~\alphathin}

\newcommand{\edgecon}{edge-con\-nect\-ed}

\usepackage{wrapfig}
\usepackage{tikz}

\title{A grid theorem for strong immersions of walls}
\author{Reinhard Diestel \and Raphael W. Jacobs \and Paul Knappe \and Paul Wollan}


\begin{document}

\begin{abstract}
    We show that a graph contains a large wall as a strong immersion minor if and only if the graph does not admit a tree-cut decomposition of small `width', which is measured in terms of its adhesion and the path-likeness of its torsos.
\end{abstract}

\maketitle

\vspace{-0.5cm}

\section{Introduction} \label{sec:Introduction}

In their Graph Minors Project~\cite{GM}, Robertson and Seymour investigated the structure of graphs not containing a fixed graph as a minor.
An important example of their structure theorems as well as a result central to their project is the \emph{grid theorem}~\cite{GMV} which asserts the equivalence of large tree-width and the existence of large grid minors.
More precisely, every graph of large enough tree-width contains a large grid as a minor, and large grid minors form an obstruction to small tree-width in that large grids have large tree-width and every graph has tree-width at least the tree-width of each of its minors.
An equivalent formulation of the grid theorem is that a graph has large tree-width if and only if it contains a large wall as a topological minor~\cite{GMXIII}.

Grid theorem-like results, which provide the equivalence of large width with respect to some width-measure and a large (often grid-like) substructure in terms of a corresponding containment relation, have since been proven in various other settings.
For example, Kreutzer and Kawarabayashi proved a directed grid theorem~\cite{DigraphGridThm}, which establishes an equivalence of large directed tree-width and the existence of butterfly minors of large directed grids, and Geelen, Kwon, McCarty and Wollan showed in~\cite{VxMinorsGrid} that a graph has large rank-width if and only if it contains a vertex-minor isomorphic to a large comparability-grid.

In this paper we are concerned with the structure of graphs not containing a large wall as a strong immersion minor.
As the minor relation, immersion minors generalise the notion of topological minors, but they are unrelated to minors in that neither the existence of an immersion minor implies the existence of a minor nor vice-versa.%
\COMMENT{}
Immersion minors have attracted attention from various algorithmic and structural viewpoints in recent years, often finding results analogous to those for minors~\cites{GMXXIII,WollanImmersion,devos2014minimum,abu2003graph,ferrara2008h,GKMW-TopSubgraphsFPT,dvorak2014strong,dvorak2016structure,Immersionminimalinfinitelyedgecongraph}.

There are two versions of the immersion relation, a weak one and a strong one.
They are defined as follows.
A \emph{weak immersion} of a graph~$H$ in a graph~$G$ is a map~$\alpha$ with domain~$V(H)\cup E(H)$ that embeds~$V(H)$ into~$V(G)$ and maps every edge~$uv\in H$ to an~$\alpha(u)$--$\alpha(v)$ path in~$G$ which is edge-disjoint from every other such path.
If, additionally, these paths have no internal vertices in~$\alpha(V(H))$, then~$\alpha$ is a \emph{strong immersion} of~$H$ in~$G$.
The vertices in~$\alpha(V(H))$ are the \emph{branch vertices} of this immersion.
We say that~$H$ is \emph{weakly/strongly immersed} in~$G$, or that~$H$ is a \emph{weak/strong immersion minor} of~$G$, if there is a weak/strong immersion of~$H$ in~$G$.

While the (topological) minor relation behaves well with tree-decompositions, the immersion relations do so with \emph{{\tcd}s}, which form the edge-cut analogue of tree-decompositions.
Formally, a pair~$(T, \cX)$ is a~\emph{{\tcd}} of a graph~$G$ if~$T$ is a tree and~$\cX = (X_t)_{t \in T}$ a near-partition of~$V(G)$, that is, the~$X_t$ are disjoint and their union is~$V(G)$ but, in contrast to a partition, the~$X_t$ may be empty.
The vertex sets~$X_t$ are the \emph{parts} of~$(T, \cX)$.
The \emph{torso of~$(T, \cX)$ at a node~$t \in T$} arises from~$G$ by identifying for every component~$T'$ of~$T - t$ the vertices in~$\bigcup_{t' \in T'} X_{t'}$ to a single vertex, keeping parallel edges as they arise but omitting loops.
These new identification vertices are the \emph{peripheral vertices} of the torso, while the ones in~$X_t$ are its \emph{core vertices}.
Every edge~$t_1 t_2 \in T$ induces its \emph{adhesion set}~$E_G(\,\bigcup_{t \in T_1} X_t\,,\,\bigcup_{t \in T_2} X_t\,)$ where~$T_1$ and~$T_2$ are the two components of~$T - t_1 t_2$ with~$t_1 \in T_1$ and~$t_2 \in T_2$ and, for $X,Y \subseteq V(G)$, $E_G(X,Y)$ denotes the set of edges $xy$ of $G$ with $x \in X$ and $y \in Y$.
The \emph{adhesion} of a {\tcd} then is the maximum size of its adhesion sets.

DeVos, McDonald, Mohar and  Scheide~\cite{devos2013note} and~Wollan~\cite{WollanImmersion} independently showed that if a graph does not admit~$K^\ell$, the complete graph of size~$\ell$, as a weak immersion, then it admits a \tcd\ each of whose torsos contains at most~$\ell$ vertices of degree at least~$\ell^2$.
As a qualitative converse, they proved that a~\tcd\ in which each torso contains at most~$\ell$ vertices of degree at least~$\ell$ precludes the existence of a weak immersion of~$K^{\ell+1}$. 
Introducing the concept of tree-cut width, Wollan~\cite{WollanImmersion} also derived a grid theorem for weak immersions, which establishes the equivalence of large tree-cut width and the existence of weak immersions of large walls.
His proof made use of the classical grid theorem for excluded wall minors by Robertson and Seymour~\cites{GMV}.

For strong immersions Dvo{\v{r}}{\'a}k and Wollan proved the following general structure theorem describing the structure of graphs not containing a fixed graph as a strong immersion minor.
\begin{theorem}[\cite{dvorak2016structure}*{Theorem~4}] \label{thm:StrongImmersionStructureThm}
    For every graph~$F$, there exists an integer~$\alpha = \alpha(F) > 0$ such that if a graph~$G$ does not contain~$F$ as a strong immersion minor, then there exists a \tcd\ of~$G$ of adhesion less than~$\alpha$ such that each of its torsos is~$\alpha$-basic.
\end{theorem}

\noindent
Here, being $\alpha$-basic for a graph means, roughly speaking, that it has a path-like decomposition in which we can accommodate vertices of high degree.

When we restrict~$F$ to complete graphs in~\cref{thm:StrongImmersionStructureThm}, then there is also a qualitative converse of~\cref{thm:StrongImmersionStructureThm} in~\cite{dvorak2016structure}:
For every integer~$\alpha > 0$, there exists an integer~$n = n(\alpha) > 0$ such that any graph with a \tcd\ which has adhesion less than~$\alpha$ and whose torsos are~$\alpha$-basic does not contain~$K^n$ as a strong immersion minor.

If we consider strong immersions of walls, as it is the subject of this paper, and we thus look at~\cref{thm:StrongImmersionStructureThm} with~$F$ restricted to walls, then~\cref{thm:StrongImmersionStructureThm} does not have such a qualitative converse.
Indeed,
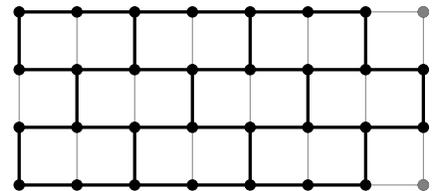
\begin{wrapfigure}{r}{.35\textwidth}
    \centering
    \vspace{-0.1cm}
    \resizebox{.35\textwidth}{!}{
        \begin{tikzpicture}[]
            \def\sizeofwalla{3}
            \def\sizeofwallb{3}
    
            \foreach \i in {2,3,...,\sizeofwallb}{        
            	\draw[ultra thick] (0,\i-1)--(2*\sizeofwalla+1,\i-1);
            }
            
            \draw[ultra thick] (0,0)--(2*\sizeofwalla,0);
            \draw[ultra thick] (0,\sizeofwallb)--(2*\sizeofwalla,\sizeofwallb);
            
            \pgfmathsetmacro{\bh}{(\sizeofwallb-1)/2};
            \foreach \i in {0,1,...,\sizeofwalla}{
            	\foreach \j in {0,1,...,\bh}{
            		\draw[ultra thick] (2*\i,2*\j)--(2*\i,2*\j+1);
            	}
            }
        
            \foreach \i in {0,1,...,\sizeofwalla}{    
            	\foreach \j in {0}{
            		\draw[ultra thick] (2*\i+1,2*\j+1)--(2*\i+1,2*\j+2);
            	}
            }
    
            \draw[thin,gray] (2*\sizeofwalla,0)--(2*\sizeofwalla+1,0);
            \draw[thin,gray] (2*\sizeofwalla,3)--(2*\sizeofwalla+1,3);
    
            \foreach \i in {0,1,...,\sizeofwalla}{    
            	\foreach \j in {0}{
            		\draw[thin,gray] (2*\i,2*\j+1)--(2*\i,2*\j+2);
            	}
            }
    
            \foreach \i in {0,1,...,\sizeofwalla}{
            	\foreach \j in {0,...,\bh}{
            		\draw[thin,gray] (2*\i+1,2*\j)--(2*\i+1,2*\j+1);
            	}
            }
            
            \pgfmathsetmacro{\ah}{2*\sizeofwalla+1};
            \foreach \i in {0,...,\ah}{
                \foreach \j in {0,...,\sizeofwallb}{
                    \draw[fill=black] (\i,\j) circle (0.09);
                }
            }
            
            \draw[gray,fill=gray] (\ah,0) circle (0.09);
            \draw[gray,fill=gray] (\ah,3) circle (0.09);
        \end{tikzpicture}
    }
    \caption{The wall~$W_4$ of size~$4$ with the underlying~$4 \times 8$~grid.}
    \label{fig:wall}
\end{wrapfigure}
for every integer~$\ell > 0$, the~\emph{wall~$W_\ell$ of size~$\ell$} is formally defined as the graph arising from the~$\ell \times 2\ell$ grid by deleting all edges~$(i, j) (i', j')$ with~$j = j'$, $i = i'+1$ and~$j \equiv i~(mod\ 2)$ and then removing the two resulting vertices of degree~$1$ (see~\cref{fig:wall}).
By definition the wall~$W_\ell$ has~$2\ell^2 - 2$ vertices and maximum degree~$3$.
Thus, $W_\ell$ does not contain~$K^5$ as a strong immersion minor.
\cref{thm:StrongImmersionStructureThm} hence yields an integer~$\alpha > 0$ such that every wall, no matter how large its size, has a \tcd\ of adhesion less than~$\alpha$ such that each torso is~$\alpha$-basic.
In particular, \cref{thm:StrongImmersionStructureThm} does not admit a qualitative converse for walls. 

In this paper, we prove a grid theorem-like result for the strong immersion relation which says that a graph contains a large wall as a strong immersion minor if and only if the graph has a {\tcd} of a specific type.\vadjust{\penalty-200}
These \tcd s are inspired by those in~\cref{thm:StrongImmersionStructureThm} in that their torsos still admit a path-like structure, but path-likeness is defined differently, as follows.

A graph~$G$ is~\emph{\alphathin} for an integer~$\alpha > 0$ if there exists an enumeration~$v_1, \dots, v_n$ of its vertices such that there are at most~$\alpha$ edges in~$G$ joining~$\{v_1, \dots, v_{i-1}\}$ and~$\{v_{i+1}, \dots, v_n\}$ for~$i = 1, \dots, n$,%
\COMMENT{}
and we call~$G$~\emph{\almostalphathin} if it becomes~\alphathin\ after removing up to~$\alpha$ vertices each of which has at most~$\alpha$ neighbours in~$G$.
In our specific \tcd s, the torsos then have~\almostalphathin\ `$3$-centres'.
Following~\cite{WollanImmersion}, the~\emph{$3$-centre} of a torso of a \tcd\ arises from the torso by a maximal sequence of deleting peripheral vertices of degree at most~$1$ and suppressing peripheral vertices of degree~$2$, removing arising loops.
We will show that if a torso of a \tcd\ is itself \alphathin, then so is its~$3$-centre, while the converse does not hold.
For a further discussion of these notions, we refer the reader to~\cref{sec:Alphathin}.

Our first main result, a version of~\cref{thm:StrongImmersionStructureThm} specifically tailored to walls, then reads as follows:
\begin{mainresult} \label{main:NoWallImpliesAlphaThin}
    For every integer~$\ell > 0$, there exists an integer~$\alpha = \alpha(\ell) > 0$ such that if a graph~$G$ does not contain the wall~$W_\ell$ as a strong immersion minor, then~$G$ has a {\tcd} of adhesion at most~$\alpha$ such that the~$3$-centres of its torsos are~\almostalphathin.
\end{mainresult}

\noindent We will see that~$3$-centres are indeed necessary in~\cref{main:NoWallImpliesAlphaThin} in that we cannot consider the torsos themselves instead of their~$3$-centres (see~\cref{ex:3CentresAreNecessary}).

In contrast to~\cref{thm:StrongImmersionStructureThm}, the notion of path-like torsos as in~\cref{main:NoWallImpliesAlphaThin} now gives the desired equivalence of the existence of strong immersions of large walls and these {\tcd}s.
More precisely, our second main result provides a qualitative converse of~\cref{main:NoWallImpliesAlphaThin}:
\begin{mainresult} \label{main:AlphaThinImpliesNoWall}
    For every integer~$\alpha > 0$, there exists an integer~$\ell = \ell(\alpha) > 0$ such that if a graph~$G$ has a {\tcd} of adhesion at most~$\alpha$ such that the~$3$-centres of its torsos are \almostalphathin, then~$G$ does not contain the wall~$W_\ell$ as a strong immersion minor.
\end{mainresult}

While the proof of~\cref{main:AlphaThinImpliesNoWall} is self-contained and does not build on any previous results, the proof of~\cref{main:NoWallImpliesAlphaThin} draws on the grid theorem for excluded wall minors as well as on several previously established methods for immersions and {\tcd}s~\cites{dvorak2016structure,WollanImmersion,GMV}.
We emphasise that we do not make use of any previous structure theorem for weak or strong immersions, but rather use some of their ideas and combine them in a novel way to fit our specific problem. \\

This paper is organised as follows.
We first discuss the notions of~\almostalphathin\ graphs and~$3$-centres in~\cref{sec:Alphathin}.
Then we prove~\cref{main:AlphaThinImpliesNoWall} in~\cref{sec:WallsAsObstructions} and~\cref{main:NoWallImpliesAlphaThin} in~\cref{sec:ExcludingWalls}. \\

For basic graph-theoretic terms, we follow~\cite{DiestelBook24}.
In this paper, we only consider graphs without loops.
But unless called \emph{simple}, they may contain parallel edges.
As a consequence, the \emph{degree} of a vertex is the number of its incident edges which may in general differ from the number of its neighbours.
Moreover, we diverge from~\cite{DiestelBook24} in that we define the \emph{components} of a graph not as its maximal connected subgraphs, but as their vertex sets.

\section{Almost \texorpdfstring{$\alpha$}{alpha}-thin graphs and~\texorpdfstring{$3$}{3}-centres} \label{sec:Alphathin}

In this section we take a closer look into our notion of path-likeness, (almost)~\alphathin\ graphs, and its relation to the concept of~$3$-centres.
Let us first recall the definition of (almost)~\alphathin ness from the introduction.
\begin{definition}
    Let~$\alpha > 0$ be an integer.
    A graph~$G$ is~\emph{\alphathin} if there exists an enumeration~$v_1, \dots, v_n$ of its vertices satisfying that~$|E_G(\{v_1, \dots, v_{i-1}\}, \{v_{i+1}, \dots, v_n\})| \le \alpha$ for~$i = 1, \dots, n$,\footnote{Note that the definition of~$\alpha$-thin differs from the similar notion of \emph{cutwidth} at most~$\alpha$, since cutwidth also takes the edges joining~$v_i$ and~$\{v_{i+1}, \dots, v_n\}$ into account. While cutwidth bounds thinness from above, there are graphs with unbounded cutwidth which are~$0$-thin, as paths with many parallel edges witness.} and the graph~$G$ is~\emph{\almostalphathin} if there exists a set~$X$ of at most~$\alpha$ vertices of~$G$ such that each vertex in~$X$ has at most~$\alpha$ neighbours in~$G$ and~$G - X$ is~\alphathin.
\end{definition}
\noindent We remark that if a graph is \alphathin\ or almost \alphathin, then so are all its subgraphs.

Even though adding vertices of degree~$1$ to a graph does not change whether it contains a wall as a strong immersion minor, such vertices affect whether a graph is \alphathin\ or almost \alphathin:

\begin{example2} \label{ex:StarVerticesDeg1}
    For every integer~$\alpha > 0$, the star~$K_{1,n}$ with~$n = 3 \alpha + 3$ leaves is not~\almostalphathin, but becomes~$0$-thin after repeatedly removing vertices of degree~$1$.
\end{example2}
\begin{proof}
    Repeatedly removing vertices of degree~$1$, the star~$G := K_{1, n}$ becomes the graph on one vertex, which is $0$-thin.
    So it remains to argue that~$G$ is not~\almostalphathin, and we suppose for a contradiction that there exists a set~$X$ of at most~$\alpha$ vertices of~$G$, each with at most~$\alpha$ neighbours in~$G$, such that~$G - X$ is~\alphathin.
    
    The set~$X$ cannot contain the centre~$c$ of~$G$ since it has~$n > \alpha$ many neighbours in~$G$.
    Thus, $X$ consists of up to~$\alpha$ leaves of~$G$, and~$G - X$ is again a star with centre~$c$, now with at least~$n' \ge 2 \alpha + 3$ leaves.
    Now take an arbitrary enumeration~$v_1, \dots, v_{n'+1}$ of~$V(G - X)$, and let~$i \in \{1, \dots, n'+1\}$ with~$v_i = c$.
    If~$i \ge \alpha + 3$, then~$\{v_1, \dots, v_{i-2}\}$ contains at least~$\alpha + 1$ leaves of the star, while if~$i \le (n' + 1) - (\alpha + 3)$, then~$\{v_{i+2}, \dots, v_{n'+1}\}$ contains at least~$\alpha + 1$ leaves.
    So for~$j = i-1$ in the first case and~$j = i+1$ in the second case, there are at least~$\alpha + 1$  edges joining~$\{v_1, \dots, v_{j-1}\}$ and~$\{v_{j+1}, \dots, v_{n'+1}\}$, which yields the desired contradiction.
\end{proof}

Similarly to the addition of vertices of degree~$1$, subdividing edges can have an impact on \alphathin ness:
\begin{example2} \label{ex:ParallelVerticesDeg2}
    For every integer~$\alpha > 0$, the graph~$G$ that arises from two disjoint stars, each with~$n = 3 \alpha + 3$ leaves, by identifying their leaves is not~\almostalphathin, but becomes~$0$-thin after repeatedly suppressing vertices of degree~$2$.
\end{example2}
\begin{proof}
    If we repeatedly suppress vertices of degree~$2$, then~$G$ turns into the graph consisting of two vertices which are joined by~$n$ edges.
    This graph is trivially~$0$-thin.
    Since~$G$ contains~$K_{1, n}$ as a subgraph, and~$K_{1,n}$ is not~\almostalphathin\ by~\cref{ex:StarVerticesDeg1}, $G$~is not \almostalphathin.
\end{proof}

Just as a graph stays \alphathin\ when we delete some of its vertices, suppressing vertices of degree~$2$ does not impact the \alphathin ness of a graph:

\begin{proposition} \label{lem:AlphaThinPrimeSubdividingSuppressing}
    Let~$G$ be an \alphathin\ graph for some integer~$\alpha > 0$.
    Then suppressing a vertex of~$G$ of degree~$2$ and deleting a potentially arising loop results in an \alphathin\ graph.
\end{proposition}%
   \COMMENT{}

\begin{proof}
    Let~$G$ be an~\alphathin\ graph, let~$v \in G$ be a vertex of degree~$2$. Let~$G'$ arise from~$G$ by suppressing~$v$, omitting a potentially arising loop.
Take any vertex enumeration of~$G$ that witnesses its \alphathin ness. The induced vertex enumeration of~$G'$ is easily seen to witnesses its \alphathin ness too.%
   \COMMENT{}
   \end{proof}

We remark that \cref{lem:AlphaThinPrimeSubdividingSuppressing} has no analogue for almost \alphathin ness, as we had incorrectly claimed in the published version of this paper and used in its proof.%
   \footnote{As a counterexample to the `almost \alphathin' version for $\alpha=2$, AI~\cite{Mathpocalypse2026} offered a~$K_{3,3}$ with one edge subdivided.}
   Our revised proof here no longer needs either \cref{lem:AlphaThinPrimeSubdividingSuppressing} or its incorrect analogue for almost \alphathin ness.

In our proof of \cref{main:NoWallImpliesAlphaThin} we shall need to be able to ignore peripheral torso vertices of degree at most~2; see~\cref{ex:3CentresAreNecessary} for why.
We formalize this by the notion of `$3$-centres' from~\cite{WollanImmersion}:

\begin{definition}
    Let~$X$ be a set of vertices of a graph~$G$.
    The~\emph{$3$-centre of~$(G, X)$} arises from~$G$ by any maximal sequence of deleting vertices of degree at most~$1$ and suppressing vertices of degree~$2$ (omitting any resulting loops) where all these vertices are not in~$X$.
    If~$(T, \cX)$ is a {\tcd} of~$G$, then the~\emph{$3$-centre of the torso~$H_t$ of~$(T, \cX)$ at the node~$t \in T$} is defined as the~$3$-centre of~$(H_t, X_t)$.
\end{definition}
 
\noindent
We remark that these 3-centres are independent of the order in which vertices are deleted or suppressed, and are hence uniquely defined~\cite{WollanImmersion}*{Lemma 9}.

\medbreak

However, the~$3$-centre of a pair~$(G, X)$ cannot always be formed by first repeatedly deleting vertices of degree~$1$ and suppressing vertices of degree~$2$ afterwards: such a sequence may fail to be maximal. This is because the deletion of loops arising in vertex suppressions may create new vertices of degree~$1$, which then have to be deleted afterwards. As an example consider two triangles joined by an edge, with~$X$ the vertex set of one of the triangles.

If $G$ is \alphathin, then the~$3$-centre of~$(G, X)$%
   \COMMENT{}
   is also \alphathin, since deleting vertices and suppressing vertices of degree~$2$ does not affect \alphathin ness (\cref{lem:AlphaThinPrimeSubdividingSuppressing}).
In particular, any \tcd\ of adhesion at most~$\alpha$ of~$G$ whose torsos are (themselves) \alphathin\ satisfies the premise of~\cref{main:AlphaThinImpliesNoWall}.%
   \COMMENT{}

As we have seen in~\cref{ex:StarVerticesDeg1} and~\cref{ex:ParallelVerticesDeg2}, the converse statement is not true:
If the~$3$-centre of~$(G, X)$ is~\alphathin, then~$G$ does not even have to be~\almostalphathin\ itself. As our next example shows, we similarly cannot replace the $3$-centres of torsos in~\cref{main:NoWallImpliesAlphaThin} by the torsos themselves:

\begin{example2}\label{ex:3CentresAreNecessary}
    For every two integers~$\ell \ge 2$ and~$\alpha > 0$, there exists a graph which does not contain the wall~$W_\ell$ as a strong immersion minor but has no {\tcd} of adhesion at most~$\alpha$ whose torsos are {\almostalphathin}.
\end{example2}
\begin{proof}
    Let~$G = K_{1, n}$ be the star with~$n = \alpha (3 \alpha + 3)$ leaves and centre~$c \in V(G)$.
    Clearly, $G$ does not contain~$W_\ell$ as a strong immersion minor since~$G$ contains only a single vertex of degree more than~$1$ while~$W_\ell$ has several of those because of~$\ell \ge 2$.
    Thus, it remains to consider an arbitrary {\tcd}~$(T, \cX)$ of~$G$ of adhesion at most~$\alpha$ and to show that at least one of the torsos of~$(T, \cX)$ is not {\almostalphathin}.

    Let~$s$ be the (unique) node of~$T$ such that~$X_s$ contains the centre~$c$ of the star~$G$.
    For each neighbour~$t$ of~$s$ in~$T$, let~$Y_t := \bigcup_{t' \in T'} X_{t'}$ where~$T'$ is the component of~$T - st$ containing~$t$.
    Since~$(T,\cX)$ has adhesion at most~$\alpha$ and~$c \in X_s$, the set~$Y_t$ contains at most~$\alpha$ leaves of~$G$.
    But~$G$ has~$\alpha (3 \alpha + 3)$ leaves, so it yields a star~$G'$ with at least~$3 \alpha + 3$ leaves in the torso~$H_s$ of~$(T, \cX)$ at~$s$.
    This star~$G'$, however, is not {\almostalphathin} by~\cref{ex:StarVerticesDeg1}, and hence~$H_s$ is not {\almostalphathin} either.
\end{proof}
\noindent If we adapt~\cref{ex:3CentresAreNecessary} by doubling each edge in~$G$, then this modified version also demonstrates that just deleting peripheral vertices of degree~$1$ in each torso instead of taking its~$3$-centre does not suffice either.

\section{Strongly immersed walls are obstructions} \label{sec:WallsAsObstructions}

In this section we prove~\cref{main:AlphaThinImpliesNoWall}, which we recall here:
\begin{customthm}{\cref{main:AlphaThinImpliesNoWall}}
    For every integer~$\alpha > 0$, there exists an integer~$\ell = \ell(\alpha) > 0$ such that if a graph~$G$ has a {\tcd} of adhesion at most~$\alpha$ such that all the~$3$-centres of its torsos are \almostalphathin, then~$G$ does not contain the wall~$W_\ell$ as a strong immersion minor.
\end{customthm}

As a key ingredient to our proof of~\cref{main:AlphaThinImpliesNoWall}, we need~\cref{lem:WellLinkedInWall}, which asserts that the wall~$W_\ell$ contains a vertex set~$Z$ of size~$\ell - 2$ which is~\emph{well-linked} in~$W_\ell$, that is, for every two disjoint~$A, B \subseteq Z$ with~$k := |A| = |B|$, there exist~$k$ vertex-disjoint~$A$--$B$ paths in~$W_\ell$.
The existence of such a set~$Z$ of size~$\ell$ follows from the facts that the tree-width of a wall of size~$\ell$ is at least~$\ell$ and the largest size of a well-linked set in a graph is lower bounded by its tree-width (see for example the survey~\cite{harvey2017parameters}).
We include a direct and constructive proof in the case of walls here which yields a further structural property of the well-linked set.

\begin{lemma}\label{lem:WellLinkedInWall}
    The wall~$W_{\ell}$ of size~$\ell \ge 3$ contains a set~$Z$ of vertices of size~$\ell-2$ such that~$Z$ is well-linked in~$W_{\ell}$ and such that every two vertices in~$Z$ are joined by three internally vertex-disjoint paths in~$W_{\ell}$.
\end{lemma}

\begin{proof}
    Let~$Z$ be the set of all vertices of~$W_\ell$ which have the form~$(i,\ell)$ and degree~$3$.
    Note that~$Z$ has size~$\ell-2$ by the definition of~$W_\ell$. 
    It is easy to see that every two vertices of degree~$3$ are joined by three internally vertex-disjoint paths in~$W_{\ell}$ as~$\ell \ge 3$.
    Thus, it remains to show that for every two disjoint subsets~$A, B$ of~$Z$ with~$k := |A| = |B|$, there exist~$k$ vertex-disjoint~$A$--$B$ paths in~$W_{\ell}$.
    By Menger's theorem~\cite{DiestelBook24}, it suffices to prove that for every set~$X \subseteq V(G)$ of size at most~$k-1$ there exists an~$A$--$B$~path in~$W_\ell$ which avoids~$X$.
    
    For every~$j = 1, \dots, \ell$ let~$V_j$ be the~$j$-th \emph{vertical path} of the wall~$W_\ell$, that is,~$V_j$ is the induced subgraph of~$W_{\ell}$ on the set of all vertices of the form~$(i,2j)$ or~$(i,2j-1)$.
    Note that these vertical paths are pairwise vertex-disjoint, that each vertical path meets~$Z$ in at most one vertex, and that each vertex in~$Z$ is contained in some vertical path.
    Thus, the pigeonhole principle yields a vertical path~$V_{j_A}$ which meets~$A$, but avoids~$X$, since~$|A| = k > k-1 \geq |X|$; analogously, we obtain a vertical path~$V_{j_B}$ for~$B$.
    
    Next, for every~$i = 1, \dots, \ell$ let~$H_j$ be the~$i$-th \emph{horizontal path} of the wall~$W_\ell$, that is, the subgraph of~$W_{\ell}$ induced on the set of all vertices of the form~$(i,j)$.
    Note that these horizontal paths are pairwise vertex-disjoint.
    Thus, the pigeonhole principle yields an horizontal path~$H_{i_0}$ which avoids~$X$, since~$\ell > |A| > |X|$.
    
    By definition, every horizontal path meets every vertical path.
    Hence,~$V_{j_A} + H_{i_0} + V_{j_B}$ is a connected subgraph of~$W_\ell$ which meets both~$A$ and~$B$, but avoids~$X$.
    In particular, there exists an~$A$--$B$~path in~$W_\ell$ which avoids~$X$, as desired.
\end{proof}

\begin{proof}[Proof of~\cref{main:AlphaThinImpliesNoWall}]
    We set
    \begin{equation*}
        \ell := \ell(\alpha) := [(\alpha^2 +1) (2(\alpha+1)+4) + \alpha^2 + \alpha] + [\alpha ((\alpha^2+1) (2 (\alpha+1) + \alpha +2) +\alpha^2 + \alpha)] +2,
    \end{equation*}
    and claim that this~$\ell$ is as desired.
    So let~$G$ be a graph, and let~$(T, \cX)$ be a \tcd\ of~$G$ of adhesion at most~$\alpha$ such that the~$3$-centres of its torsos are~\almostalphathin.
    Suppose for a contradiction that the wall~$W_\ell$ is strongly immersed in~$G$, and let~$U$ be the set of branch vertices of this strong immersion.
    
    By~\cref{lem:WellLinkedInWall}, $W_{\ell}$ contains a vertex set~$Z$ of size at least~$\ell-2$ such that~$Z$ is well-linked in~$W_{\ell}$ and every two vertices in~$Z$ are joined by three internally vertex-disjoint paths in~$W_{\ell}$.
    Let~$Z_U \subseteq U$ be the set of branch vertices corresponding to~$Z \subseteq V(W_\ell)$ in~$G$ via the strong immersion.
    Since~$Z$ is well-linked in~$W_\ell$, the definition of strong immersions yields that this set~$Z_U$ has the following property:
    \vspace{.5\baselineskip}
    \begin{equation*}
        \tag{$\ast$}
        \begin{aligned} \label{equ:PropStar}
            \parbox{\textwidth-3\parindent}{\emph{For every two disjoint sets~$A, B \subseteq Z_U$ with~$k := |A| = |B|$, there exist~$k$ edge-disjoint~$A$--$B$~paths in~$G$ and these~$k$ paths meet each vertex of~$U$ at most once.}}
        \end{aligned}
    \end{equation*}
        
    \noindent Moreover, since every two vertices in~$Z$ are joined by three internally vertex-disjoint paths in~$W_{\ell}$, every two vertices in~$Z_U$ are joined by three edge-disjoint paths in~$G$.

    Consider an edge~$t_1 t_2 \in T$ and let~$Y_1 := \bigcup_{t \in T_1} X_{t}$ where~$T_1$ is the component of~$T - t_1 t_2$ containing~$t_1$, and define~$Y_2$ analogously.
    Since~$(T, \cX)$ has adhesion at most~$\alpha$, the adhesion set~$E_G(Y_1, Y_2)$ of~$G$ induced by~$t_1 t_2$ has size at most~$\alpha$.
    Therefore, precisely one~$Y_i$, say~$Y_2$, contains at least~$\alpha+1$ vertices from~$Z_U$ by~\cref{equ:PropStar} and since~$\ell -2 \geq 2 (\alpha+1)$.%
    \COMMENT{}
    We then orient the edge~$t_1t_2$ from~$t_1$ to~$t_2$.\vadjust{\penalty-200}
    In this way, $Z_U$ induces an orientation of the edges of~$T$ such that for each node~$t \in T$ at most one incident edge is oriented away from~$t$.\looseness=-1

    Hence, there is a unique sink~$s$ of this orientation, that is, a node~$s \in T$ such that all incident edges are oriented towards~$s$.
    Let~$H$ be the torso of~$(T, \cX)$ at~$s$.
    For each neighbour~$t$ of~$s$ in~$T$, let~$Y_t := \bigcup_{t' \in T'} X_{t'}$ where~$T'$ is the component of~$T - st$ containing~$t$, and let~$v_t$ be the peripheral vertex of~$H$ corresponding to the identification of the vertices in~$Y_t$.
    
    Consider the set~$Z_H$ of vertices of~$H$ corresponding to the vertices in~$Z_U$, that is, the set consisting of the core vertices~$Z_U \cap X_s$ together with those peripheral vertices~$v_t$ with~$Z_U \cap Y_t \neq \emptyset$.
    This set~$Z_H$ is then contained in the~$3$-centre~$\bar{H}$ of the torso~$H$.
    Indeed, for every peripheral vertex~$v_t \in Z_H$, the set~$Z_U$ contains a vertex in~$Y_t$ by the definition of~$Z_H$. 
    Since~$st$ is oriented towards~$s$, the set~$Z_U$ also contains a vertex which is not in~$Y_t$.
    But these two vertices, just as any two vertices in~$Z_U$, are joined by three edge-disjoint paths in~$G$.
    Thus, $v_t$ is never a candidate for deletion or suppression in the construction of~$\bar{H}$ from~$H$.
    
    For every peripheral vertex~$v_t$ in~$Z_H$, we have~$|Z_U \cap Y_t| \le \alpha$, since~$st$ was oriented towards~$s$.
    This implies that~$|Z_H| \ge |Z_U \cap X_s| + |Z_U \setminus X_s|/\alpha$.
    In particular, $Z_H$ contains either many core vertices or many peripheral vertices of~$H$.
    
    We now aim to use the set~$Z_H$ to derive a contradiction to the~$3$-centre~$\bar{H}$ of the torso~$H$ being \almostalphathin.
    To do so, consider a set~$X$ of at most~$\alpha$ vertices of~$\bar{H}$ each of which has at most~$\alpha$ neighbours in~$\bar{H}$ and an enumeration~$v_1, \dots, v_n$ of~$V(\bar{H} - X)$ which witnesses that~$\bar{H} - X$ is~\alphathin.
    Observe that if we remove the at most~$\alpha^2$ neighbours of the set~$X$ among the~$v_i$, then this partitions the enumeration~$v_1, \dots, v_n$ into a set~$\cI$ of at most~$\alpha^2+1$ intervals.
    
    If~$Z_H$ contains many core vertices of~$H$ in that~$|Z_H \cap X_s| \geq (\alpha^2 +1) (2(\alpha+1)+4)+\alpha^2 + \alpha$, then~$\{v_1, \dots, v_n\}$ contains at least~$(\alpha^2 +1) (2(\alpha+1)+4)$ core vertices in~$Z_H$ which are not neighbours of~$X$.
    Thus, the pigeonhole principle yields an interval~$v_j, \dots, v_k$ in~$\cI$ containing at least~$2(\alpha+1)+4$ elements of~$Z_H \cap X_s$; shortening the interval if necessary, we may assume additionally that it contains precisely~$2(\alpha+1)+4$ elements of~$Z_H \cap X_s$ and that its endpoints~$v_j$ and~$v_k$ are among those.
    
    Let~$A_H$ be the set consisting of the~$2(\alpha+1)+2$ elements of~$Z_H \cap X_s$ in the interior~$\{v_{j+1}, \dots, v_{k-1}\}$ of our fixed interval, and let~$B_H$ be a subset of~$Z_H \cap X_s$ of the same size as~$A_H$ and which is disjoint from~$\{v_j, \dots, v_k\}$; such a set~$B_H$ exists since~$Z_H \cap X_s$ is large enough.
    Now~$A_H$ and~$B_H$ are subsets of~$Z_H$ consisting of core vertices of~$H$, and hence~$A_H, B_H \subseteq Z_U$.
    So we may apply~\cref{equ:PropStar} to~$A_H, B_H \subseteq Z_U$ and obtain~$2(\alpha+1)+2$ edge-disjoint~$A_H$--$B_H$ paths in~$G$ such that at most two of these paths meet~$v_j$ or~$v_k$, since~$v_j$ and~$v_k$ are contained in~$Z_H \cap X_s \subseteq Z_U \subseteq U$.

    Since the torso~$H$ arises from~$G$ by identifications of vertices keeping parallel edges as they arise, these~$2(\alpha+1)+2$ paths induce~$2(\alpha+1)+2$ edge-disjoint~$A_H$--$B_H$ walks in~$H$ which we can shorten to~$2 (\alpha+1)+2$ edge-disjoint~$A_H$--$B_H$ paths in~$H$.
    By the definition of~$3$-centres, these~$2(\alpha+1)+2$ paths in the torso~$H$ in turn induce~$2(\alpha+1)+2$ edge-disjoint~$A_H$--$B_H$ paths in the~$3$-centre~$\bar{H}$ of~$H$.%
    \COMMENT{}
    We remark that throughout this process we maintain the property that at least~$2(\alpha+1)$ of these paths avoid~$v_j$ and~$v_k$.

    Since none of the vertices~$v_j, \dots, v_k$ is a neighbour of~$X$, these~$2 (\alpha + 1)$ edge-disjoint paths joining $\{v_{j+1}, \dots, v_{k-1}\}$ and~$\{v_1, \dots, v_{j-1}\} \cup \{v_{k+1}, \dots, v_n\}$ indeed contain at least~$2(\alpha + 1)$ edges in~$\bar{H} - X$ joining $\{v_{j+1}, \dots, v_{k-1}\}$ and~$\{v_1, \dots, v_{j-1}\} \cup \{v_{k+1}, \dots, v_n\}$.
    In particular, there are either at least~$\alpha + 1$ edges joining~$\{v_1, \dots, v_{j-1}\}$ and~$\{v_{j+1}, \dots, v_n\}$ or at least~$\alpha + 1$ edges joining $\{v_1, \dots, v_{k-1}\}$ and~$\{v_{k+1}, \dots, v_n\}$.
    In both cases we obtain a contradiction to the choice of the enumeration~$v_1, \dots, v_n$ to witness the~\alphathin ness of~$\bar{H} - X$.

    If~$Z_H$ does not contain many core vertices as in the above case, then
 $$|Z_H| \ge |Z_U \cap X_s| + |Z_U \setminus X_s|/\alpha$$
  implies that~$Z_H$ contains many peripheral vertices in that
 $$|Z_H \setminus X_s| \geq (\alpha^2+1) (2 (\alpha+1) + \alpha +2) +\alpha^2 + \alpha$$
  by our choice of~$\ell$.
    We now proceed analogously to the above case of many core vertices in~$Z_H$.
    
    By the lower bound on the size of~$Z_H \setminus X_s$, there are at least~$(\alpha^2+1) (2 (\alpha+1) + \alpha +2)$ peripheral vertices in~$Z_H$ which are neither neighbours of~$X$ nor contained in~$X$.
    So the pigeonhole principle yields an interval~$v_j, \dots, v_k$ in~$\cI$ which contains at least~$2 (\alpha+1) + \alpha+2$ elements of~$Z_H \setminus X_s$.
    We may assume additionally, by shortening the interval if necessary, that the interval contains precisely~$2 (\alpha+1) + \alpha+2$ elements of~$Z_H \setminus X_s$ and that its endpoints~$v_j$ and~$v_k$ are among those.
    
    Let~$A_H$ be the set of the~$2 (\alpha+1) + \alpha$ elements of~$Z_H \setminus X_s$ in the interior~$\{v_{j+1}, \dots, v_{k-1}\}$ of our fixed interval, and let~$B_H$ be a set of~$2 (\alpha+1) + \alpha$ elements of~$Z_H \setminus X_s$ that are not in~$\{v_{j}, \dots, v_{k}\}$; such a set~$B_H$ exists since~$Z_H \setminus X_s$ is large enough.
    Since each peripheral vertex~$v_t$ in~$Z_H$ satisfies~$Z_U \cap Y_t \neq \emptyset$, we may replace each element~$v_t$ of~$A_H$ and~$B_H$ by an arbitrary vertex in~$Z_U \cap Y_t$ to obtain subsets~$A_U$ and~$B_U$ of~$Z_U$, respectively, which are disjoint and of the same size as~$A_H$ and~$B_H$ by the disjointness of the~$Y_t$.
    Applying~\cref{equ:PropStar} to~$A_U, B_U \subseteq Z_U$ then yields~$2 (\alpha+1) + \alpha$ edge-disjoint~$A_U$--$B_U$ paths in~$G$.
    
    The construction of~$A_U$ and~$B_U$ now guarantees that these~$2 (\alpha+1) + \alpha$ many~$A_U$--$B_U$~paths in~$G$ induce~$2 (\alpha+1) + \alpha$ edge-disjoint~$A_H$--$B_H$ paths in~$\bar{H}$ as above.
    Since~$(T, \cX)$ has adhesion at most~$\alpha$ and~$v_j, v_k \in Z_H \setminus X_s$ are neither contained in~$A_H$ nor in~$B_H$, each of the peripheral vertices~$v_j$ and~$v_k$ lies on at most~$\alpha/2$ of these paths.%
    \COMMENT{}
    Thus, there are at least~$2 (\alpha+1) + \alpha - 2 \cdot \alpha/2 = 2 (\alpha+1)$ edge-disjoint paths in~$\bar{H}$ joining~$\{v_{j+1}, \dots, v_{k-1}\}$ and~$\{v_1, \dots, v_{j-1}\} \cup \{v_{k+1}, \dots, v_n\}$ which avoid~$v_j$ and~$v_k$.
    Since~$\{v_j, \dots, v_k\}$ contains no neighbours of~$X$, these paths contain at least~$2 (\alpha+1)$ edges in~$\bar{H} - X$ joining~$\{v_{j+1}, \dots, v_{k-1}\}$ and~$\{v_1, \dots, v_{j-1}\} \cup \{v_{k+1}, \dots, v_n\}$.
    Thus, there are either at least~$\alpha+1$ edges joining~$\{v_{j+1}, \dots,v_{k-1}\}$ and~$\{v_1, \dots, v_{j-1}\}$ or at least~$\alpha+1$ joining~$\{v_{j+1}, \dots, v_{k-1}\}$ and~$\{v_{k+1}, \dots, v_n\}$.
    In particular, there are either at least~$\alpha+1$ edges joining~$\{v_{j+1}, \dots, v_{n}\}$ and~$\{v_1, \dots, v_{j-1}\}$ or at least~$\alpha +1$ edges joining~$\{v_{1}, \dots, v_{k-1}\}$ and~$\{v_{k+1}, \dots, v_n\}$, a contradiction.
\end{proof}

\section{Excluding strongly immersed walls} \label{sec:ExcludingWalls}

This section is devoted to the proof of~\cref{main:NoWallImpliesAlphaThin}, recalled here:

\begin{customthm}{\cref{main:NoWallImpliesAlphaThin}}
    For every integer~$\ell > 0$, there exists some integer~$\alpha = \alpha(\ell) > 0$ such that if a graph~$G$ does not contain the wall~$W_\ell$ as a strong immersion minor, then~$G$ has a {\tcd} of adhesion at most~$\alpha$ such that the~$3$-centres of its torsos are~\almostalphathin.
\end{customthm}

For the proof of~\cref{main:NoWallImpliesAlphaThin}, we first reduce the problem to~$3$-edge-connected graphs in~\cref{subsec:Reduction3Connected}.
In~\cref{subsec:immersionandtcs} we collect some tools from previous research dealing with immersions and \tcd s.
We then prove~\cref{main:NoWallImpliesAlphaThin} in~\cref{subsec:proofofnowallimpliesalphathin}.

\subsection{Reduction to 3-edge-connected graphs} \label{subsec:Reduction3Connected}

Our first aim is to reduce \cref{main:NoWallImpliesAlphaThin} to $3$-edge-connected graphs. We shall do this separately for each integer~$\ell$ from the theorem. So, throughout this section, let~$\ell > 0$ be an arbitrary, but fixed, integer.

We first record a single reduction step that removes a vertex of degree at most~2.

\begin{lemma} \label{lem:ReductionDegree3}
    Let $\alpha\ge 2$ be an integer, let $G$ be a graph of order at least~2, and let $v\in G$ be a vertex of degree $d(v)\le 2$ in~$G$. Let $G'$ arise from~$G$ by deleting~$v$ if $d(v)\le 1$, and by suppressing~$v$ (omitting any arising loop) if $d(v)=2$. If $G'$ has a \tcd\ of adhesion at most~$\alpha$ such that the 3-centres of its torsos are almost \alphathin, then so does~$G$.
\end{lemma}

\begin{proof}
    Let~$(T', \cX')$ with $\cX'={(X'_t)}_{t\in T}$ be a \tcd\ of~$G'$ of adhesion at most~$\alpha$ such that the~$3$-centres of its torsos are~\almostalphathin.
    If~$d(v) \ge 1$, let~$u$ be a neighbour of~$v$ in~$G$; otherwise let~$u$ be an arbitrary vertex of~$G$ other than~$v$. If $d(v)=2$, let $w$ be the other end of the second edge at~$v$, where possibly $w=u$.

There exists a unique node~$t_u$ of~$T'$ such that~$u \in X'_{t_u}$.
    Let $(T, \cX)$ with $\cX = (X_t)_{t\in T}$ be the \tcd\ of~$G$ obtained from~$(T', \cX')$ in which $T$ arises from~$T'$ by joining a new vertex~$t_v$ to~$t_u$, with $X_{t_v} := \{v\}$ and $X_t := X'_t$ for all~$t\ne t_v$.
    We claim that~$(T, \cX)$ is as desired.
    
    We first prove that~$(T, \cX)$ has adhesion at most~$\alpha$.
    The edge~$t_u t_v \in T$ induces an adhesion set of size at most~$2 \le \alpha$.
    Every other edge~$e$ of~$T$ is also an edge of~$T'$, and the adhesion set of~$(T, \cX)$ induced by~$e$ contains the same edges as that of~$(T', \cX')$ induced by~$e$ -- except that if~$v$ has a second neighbour~$w \neq u$ in~$G$, then the edge of~$G'$ that arose by suppressing~$v$ is replaced by the edge joining~$v$ and~$w$ in~$G$.%
    \COMMENT{}
    This replacement, however, leaves the size of the adhesion set unchanged.
    Hence, the adhesion of~$(T, \cX)$ is at most~$\alpha$, since~$(T', \cX')$ has adhesion at most~$\alpha$.
    
    It remains to show that the~$3$-centres of the torsos of~$(T, \cX)$ are~\almostalphathin.
    Since~$v$ has degree at most~$2$ in~$G$, the~$3$-centre of the torso of~$(T, \cX)$ at~$t_v$ consists only of the vertex~$v$ and is hence~\almostalphathin. For a node~$t \in T'$ other than~$t_u$, the torso of~$(T, \cX)$ at~$t$ equals the torso of~$(T', \cX')$ at~$t$: indeed, $v$~lies in the same component of~$T - t$ as~$u$, so any edge of~$G$ joining~$v$ and~$u$ becomes a loop and is omitted, while if $d(v) = 2$ and~$w \neq u$, the edge~$vw$ of~$G$ and the edge of~$G'$ which arose from suppressing~$v$ give rise to the same edge of the torso.
        Finally, let~$H$ and~$H'$ be the torsos of~$(T, \cX)$ and of~$(T', \cX')$ at~$t_u$, respectively. Note that both have~$X_{t_u} = X'_{t_u}$ as their set of core vertices.
        Then~$H$ arises from~$H'$ by adding the new peripheral vertex~$v$. This is either deleted or suppressed when the 3-centre of~$H$ is formed, which therefore equals that of~$H'$ and thus is almost \alphathin\ by assumption.%
   \COMMENT{}
\end{proof}

\begin{lemma}\label{lem:Reduction3EdgeConn}
    If \cref{main:NoWallImpliesAlphaThin} holds for all~$3$-edge-connected graphs, then it holds for all graphs.
\end{lemma}

\begin{proof}
    Let~$\alpha = \alpha(\ell) \ge 2$ be an integer such that~\cref{main:NoWallImpliesAlphaThin} holds for all~$3$-{\edgecon} graphs.
        We show that~\cref{main:NoWallImpliesAlphaThin} holds with the same~$\alpha$ for all graphs~$G$, by induction on~$|G|$.
        So let~$G$ be any graph which does not contain the wall~$W_\ell$ as a strong immersion minor, and assume that the assertion holds for all graphs with fewer vertices than~$G$.

        If~$|G| \le 1$, then the {\tcd} whose tree consists of a single node~$t$ with~$X_t := V(G)$ is as desired.
        If~$G$ has a vertex~$v$ of degree $d(v)\le 2$, let~$G'$ arise from~$G$ by deleting~$v$ if~$d(v) \le 1$, and by suppressing~$v$ (omitting any arising loops) if~$d(v) = 2$.
        Then~$G'$ does not contain~$W_\ell$ as a strong immersion minor, since $G$ does not, and~$|G'| < |G|$.
        So the induction hypothesis together with~\cref{lem:ReductionDegree3} yields a {\tcd} of~$G$ as desired.
        Finally, if~$G$ is~$3$-{\edgecon}, then such a {\tcd} exists by the choice of~$\alpha$.

        So we may assume that $G$ has minimum degree~$\delta(G)\ge 3$ but is not~$3$-{\edgecon}.
        Let~$\{A, B\}$ be a bipartition of~$V(G)$ into two non-empty sets such that the size~$c$ of the cut~$E_G(A,B)$ of~$G$ is minimum among all such bipartitions of~$V(G)$.
        Since~$G$ is not~$3$-{\edgecon}, we have~$c \le 2$. As $\delta(G)\ge 3$, this implies that both~$A$ and~$B$ have size at least~$2$.
        If~$c \ge 1$, let~$e_1 := a_1 b_1$ be an edge in~$E_G(A,B)$ with~$a_1 \in A$ and~$b_1 \in B$, and if~$c = 2$, let~$e_2 := a_2 b_2$ be the other such edge.

        We now define two smaller graphs to which we shall apply the induction hypothesis. Let
        \begin{equation*}
            G^A := \begin{cases}
                G[A] + a_1a_2 & \text{if } c = 2 \text{ and } a_1 \neq a_2,\\
                G[A] & \text{otherwise,}
            \end{cases}
        \end{equation*}
and define~$G^B$ analogously. Alternatively, $G^A$~arises from~$G$ by first contracting~$B$ to a single vertex~$b$, and then deleting~$b$ if it has degree at most~1 and suppressing it if it has degree~2, and similarly for~$G^B$.

Since $G[A]$ and~$G[B]$ are connected, by the minimality of the cut~$E(A,B)$, neither~$G^A$ nor~$G^B$ contains the wall~$W_\ell$ as a strong immersion minor: their additional edge, if~any, could be replaced by a path through the other side to yield a $W_\ell$ immersion minor of~$G$, a contradiction. By the induction hypothesis, therefore, they have {\tcd}s~$(T^A, \cX^A)$ and~$(T^B, \cX^B)$  of adhesion at most~$\alpha$ and such that the~$3$-centres of their torsos are~\almostalphathin. As usual, we write $\cX^A =: (X^A_t)_{t\in T^A}$ and $\cX^B =: (X^B_t)_{t\in T^B}$.

Let us combine these to a {\tcd}~$(T, \cX)$ of~$G$.
        If~$c \ge 1$, let~$t^A \in T^A$ be the unique node with~$a_1 \in X^A_{t^A}$, and if~$c = 0$, let~$t^A \in T^A$ be any node with~$X^A_{t^A} \ne \emptyset$.
        Define~$t^B \in T^B$ analogously for~$b_1$.
        Let~$T$ be obtained from the disjoint union of~$T^A$ and~$T^B$ by joining~$t^A$ to~$t^B$, and let $X_t := X^A_t$ for all~$t \in T^A$ and~$X_t := X^B_t$ for all~$t \in T^B$. It is readily checked that $(T, \cX)$ again has adhesion at most~$\alpha$.%
   \COMMENT{}

To complete our proof, we now show that the~$3$-centre of each torso of~$(T, \cX)$,%
   \COMMENT{}
  at $t\in T$ say, equals the $3$-centre of the torso of $(T^A, \cX^A)$ or~$(T^B,\cX^B)$ at~$t$. It will therefore be almost \alphathin\ by induction. 

Let us assume that $t \in T^A$; the case of $t\in T^B$ is analogous. Let~$H$ and~$H^A$ be the torsos of~$(T, \cX)$ and of~$(T^A, \cX^A)$ at~$t$, respectively. Both have~$X_t = X^A_t$ as their set of core vertices. If $t \ne t^A$ then $H=H^A$, since if $a_2 b_2\in E(G)$ gives rise to an edge of~$H$ then $a_2 a_1$ gives rise to an edge of~$H^A$, and conversely. In particular, $H$~and~$H^A$ have the same 3-centre.%
   \COMMENT{}

It remains to consider $t = t^A$. Let $h$ be the peripheral vertex of~$H$ that represents the component~$T^B$ of~$T-t$. By the choice of~$\{A,B\}$, this $h$ has degree at most~2 in~$H$. As $t=t^A$ now, $H^A$ arises from~$H$ by deleting or suppressing~$h$.%
   \COMMENT{}
   Since, when computing the 3-centre of~$H^A$, we may start by deleting or suppressing~$h$, the result we obtain will also be the 3-centre of~$H$.%
   \COMMENT{}
   \end{proof}

\subsection{Immersions and \tcd s}\label{subsec:immersionandtcs}

In this section we collect some tools and auxiliary results for the proof of~\cref{main:NoWallImpliesAlphaThin}.

The first lemma asserts that the absence of a large wall as a strong immersion minor implies a bound on the number of neighbours of every vertex.
This is the lemma for whose application in the proof of~\cref{main:NoWallImpliesAlphaThin} we reduced to~$3$-edge-connected graphs in~\cref{subsec:Reduction3Connected}.
\newcommand{\maxneighbourhood}[1]{g({#1})}%
\begin{lemma}[\cite{dvorak2014strong}*{Corollary~1.3}] \label{cor:DegreeIsBounded}
    For every integer~$\ell > 0$, there is an integer~$\maxneighbourhood{\ell} > 0$ such that
    if a~$3$-edge-connected graph~$G$ contains a vertex with at least~$\maxneighbourhood{\ell}$ neighbours, then the wall~$W_\ell$ is a strong immersion minor of~$G$.
\end{lemma}
\COMMENT{}%
\COMMENT{}

The next two lemmas again exclude specific substructures in the absence of a large wall as a strong immersion minor.
First, for every two integers~$k, n > 0$, let~$S_{k, n}$ be the graph on~$n+1$ vertices~$x, v_1, \dots, v_n$ with~$k$ parallel edges joining~$x$ and~$v_i$ for~$i = 1, \dots, n$.
For every graph~$G$ on~$n$ vertices which has maximum degree~$k$, one can greedily construct a strong immersion of~$G$ in~$S_{k,n}$ (cf.~\cite{WollanImmersion}*{Observation 1}).
\begin{lemma} \label{lem:SpiderContainsAll}
    For every two integers~$k, n > 0$, the graph~$S_{k, n}$ contains every graph on~$n$ vertices with maximum degree~$k$ as a strong immersion minor.
    In particular, $S_{3, 2 \ell^2}$ contains the wall~$W_\ell$ as a strong immersion minor for every integer~$\ell > 0$. \qed%
    \COMMENT{}
\end{lemma}

Secondly, let~$P_n * v$ denote the graph arising from the path~$P_n$ of length~$n$ by adding a vertex~$v$ that is adjacent to every vertex of~$P_n$.
\begin{lemma}\label{lem:PathWithApexContainsWall}
    For every integer~$ k > 0$, the graph~$P_{3k - 1} * v$ contains a subdivision of~$S_{3, k}$.
    In particular, the wall~$W_\ell$ is a strong immersion minor of~$P_{6\ell^2 -1} * v$ for every integer~$\ell > 0$.
\end{lemma}
\begin{proof}
    Deleting every third edge on the path~$P_{3k-1}$, the graph~$P_{3k - 1} * v$ becomes a subdivision of~$S_{3,k}$.
    The observation on $W_\ell$ follows immediately from~\cref{lem:SpiderContainsAll}.
\end{proof}

The following results together show that the absence of a subdivision of the wall~$W_\ell$, and thus in particular the absence of~$W_\ell$ as a strong immersion minor, implies the existence of a {\tcd} whose adhesion and torso size are both bounded in terms of the maximum degree and the size~$\ell$ of the wall.\looseness=-1

\newcommand{\treewidthbound}[1]{w({#1})}%
\begin{theorem}[Grid Theorem, \cite{GMV}*{(2.1)}]\label{lem:NoWallAsSubdivisionYieldsBoundedTW}
    For every integer~$\ell > 0$, there exists an integer~$\treewidthbound{\ell} > 0$ such that if a graph~$G$ does not contain a subdivision of the wall~$W_\ell$, then~$G$ has tree-width at most~$\treewidthbound{\ell}$.
\end{theorem}%
\COMMENT{}

\begin{lemma}[\cite{WollanImmersion}*{Lemma 12}]\label{lem:BoundedAdhesionAndTorsosViaTreewidthAndDegree}
    Let~$w, d > 0$ be integers, and let~$G$ be a graph with maximum degree at most~$d$ and tree-width at most~$w$.
    Then there exists a \tcd\ of~$G$ of adhesion at most~$(2w + 2)d$ such that every torso has at most~$(d + 1)(w + 1)$ vertices.
\end{lemma}

The next lemma says that a large enough set~$U$ of vertices in a connected graph is either reflected in a subdivision of a star with many leaves in~$U$ or a comb with many teeth in~$U$. 
Here, a \emph{comb} is a union of a path~$P$ with disjoint, possibly trivial, paths which have precisely their first vertex on~$P$, and we call the last vertices of these paths their \emph{teeth}.
Note that if a graph has small maximum degree, then the graph does not contain any large subdivided star.

\newcommand{\starcomb}[2]{h({#1},{#2})}%
\begin{lemma}[\cite{burger2022duality}*{Lemma 6.2}] \label{lem:FiniteStarComb}
    For every two integers~$s, t > 0$, there exists an integer~$\starcomb{s}{t} > 0$ such that the following holds:
    Whenever a set~$U$ of vertices of a connected graph~$G$ has size at least~$\starcomb{s}{t}$, then there is either some subdivision of a star in~$G$ with all its~$s$ leaves in~$U$ or a comb in~$G$ with all its~$t$ teeth in~$U$.
\end{lemma}%
\COMMENT{}%
\COMMENT{}

As a final ingredient to our proof of~\cref{main:NoWallImpliesAlphaThin}, the following lemma asserts that if a simple graph does not contain a star as a minor, then it is `close' to a union of vertex-disjoint paths.
\begin{lemma}[\cite{marx2014immersions}*{Lemma 20}] \label{lem:ExcludedStar}
    Let~$k > 0$ be an integer.
    If a simple connected graph~$G$ does not contain~$K_{1, k}$ as a minor, then~$G - X$ is a vertex-disjoint union of paths for some set~$X$ of at most~$4k$ vertices of~$G$.
\end{lemma}

\subsection{Proof of~\texorpdfstring{\cref{main:NoWallImpliesAlphaThin}}{Theorem 1}}\label{subsec:proofofnowallimpliesalphathin}

\begin{proof}[Proof of~\cref{main:NoWallImpliesAlphaThin}]
    We will define the integer~$\alpha(\ell) > 0$ explicitly in terms of~$\ell$ in what follows.
    Let~$G$ be a graph in which the wall~$W_\ell$ is not strongly immersed.
    By~\cref{lem:Reduction3EdgeConn}, we may assume that~$G$ is~$3$-{\edgecon}.    
    
    To construct the desired {\tcd}~$(T, \cX)$ of~$G$, we will first apply~\cref{lem:BoundedAdhesionAndTorsosViaTreewidthAndDegree} to an appropriate minor~$G'$ of~$G$ in order to obtain an auxiliary {\tcd}~$(T, \cX')$ of~$G'$ from which we then construct~$(T, \cX)$.
    This minor~$G'$ of~$G$ is defined as follows:
    \newcommand{\paralleledges}[1]{p({#1})}%
    Let~$A$ be the simple graph on~$V(G)$ with an edge joining two vertices if there are more than~$\paralleledges{\ell} := 6 \ell^2$ edges joining them in~$G$.
    The graph~$G'$ then arises from~$G$ by contracting the components of~$A$, keeping parallel edges as they arise (omitting potentially arising loops).

    Now~$G'$ inherits two key properties of~$G$, which we will need in what follows.
    First, $G'$ is still~$3$-{\edgecon} as it is a minor of the~$3$-{\edgecon} graph~$G$.
    Secondly, $G'$ does still not contain the wall~$W_\ell$ as a strong immersion minor.
    Indeed, $G$ does not contain~$W_\ell$ as a strong immersion minor, and since~$\paralleledges{\ell} \geq 3$ and~$W_\ell$ has maximum degree~$3$, one could greedily construct a strong immersion of~$W_\ell$ in~$G$ from a strong immersion of~$W_\ell$ in~$G'$.%
    \COMMENT{}
    
    In order to apply \cref{lem:BoundedAdhesionAndTorsosViaTreewidthAndDegree} to~$G'$, we next bound its tree-width and its maximum degree in terms of~$\ell$.
    \begin{claim} \label{cl:TreewidthBound}
        The tree-width of~$G'$ is at most~$\treewidthbound{\ell}$, the integer given by~\cref{lem:NoWallAsSubdivisionYieldsBoundedTW}.
    \end{claim}
    \begin{claimproof}
        Since~$G'$ does not contain~$W_\ell$ as a strong immersion minor, in particular it does not contain a subdivision of~$W_\ell$.
        Therefore,~\cref{lem:NoWallAsSubdivisionYieldsBoundedTW} yields the desired upper bound~$\treewidthbound{\ell}$ on the tree-width of~$G'$.
    \end{claimproof}
    
    \newcommand{\maxdegreebound}[1]{d({#1})}
    \begin{claim} \label{cl:MaxDegreeBound}
        There exists an integer~$\maxdegreebound{\ell}$ such that~$G'$ has maximum degree~$\maxdegreebound{\ell}$.
    \end{claim}
    \begin{claimproof}
        We prove that~$G'$ has maximum degree at most~$\maxdegreebound{\ell} := \maxneighbourhood{\ell} \cdot \paralleledges{\ell} \cdot \starcomb{2 \ell^2}{6\ell^2}^2$, where~$\starcomb{2 \ell^2}{6\ell^2}$ and~$\maxneighbourhood{\ell}$ are the integers given by~\cref{lem:FiniteStarComb} and~\cref{cor:DegreeIsBounded}, respectively.
        Since~$G'$ is~$3$-edge-connected and does not contain~$W_\ell$ as strong immersion minor, \cref{cor:DegreeIsBounded} implies that every vertex of~$G'$ has less than~$\maxneighbourhood{\ell}$ neighbours in~$G'$.
        Thus, it is enough to show that there are at most~$\maxdegreebound{\ell} / \maxneighbourhood{\ell}$ many parallel edges joining any fixed pair of vertices.
        By the construction of~$G'$, this is the case if there are at most~$\maxdegreebound{\ell} / \maxneighbourhood{\ell}$ edges in~$G$ joining any two components of~$A$.
        
        Suppose for a contradiction that there are two components~$C$ and~$C'$ of~$A$ such that there are more than~$\maxdegreebound{\ell} / \maxneighbourhood{\ell} = \paralleledges{\ell} \cdot \starcomb{2 \ell^2}{6\ell^2}^2$ edges of~$G$ joining them.
        By the definition of~$A$, each vertex in~$C$ is joined to each vertex in~$C'$ by at most~$\paralleledges{\ell}$ parallel edges in~$G$.
        Thus, at least one of~$N_G(C) \cap C'$ and~$N_G(C') \cap C$ is of size at least~$\starcomb{2 \ell^2}{6\ell^2}$, and we may assume without loss of generality that it is~$N_G(C') \cap C$.
        
        Now the simple graph~$A$ has maximum degree less than~$2 \ell^2$.
        Indeed, since~$\paralleledges{\ell} \geq 3$, a vertex of degree at least~$2 \ell^2$ in~$A$ would immediately imply the existence of a strong immersion of~$W_\ell$ in~$G$ by~\cref{lem:SpiderContainsAll}.%
        \COMMENT{}
        Thus, by~\cref{lem:FiniteStarComb}, $A[C]$ contains a comb with at least~$6 \ell^2$ teeth in~$N_G(C') \cap C$.
        
        Hence, the wall~$W_\ell$ is strongly immersed in~$G$ by~\cref{lem:PathWithApexContainsWall}, since~$G$ contains~$P_{6 \ell^2 - 1} * v$ as a strong immersion minor: 
        Fix a vertex~$v$ in~$C'$.
        Since~$\paralleledges{\ell} \geq 6 \ell^2$, we may pick greedily edge-disjoint paths from~$v$ to each tooth in~$N_G(C') \cap C$ such that all edges of each path except its last edge are contained in~$G[C']$.
        These paths together with the comb form a strong immersion of~$P_{6 \ell^2 - 1} * v$ in~$G$.%
        \COMMENT{}
        
        All in all, the maximum degree of~$G'$ is at most~$\maxdegreebound{\ell}$, as desired.
    \end{claimproof}

    \newcommand{\adhesionbound}[1]{a({#1})}%
    \newcommand{\torsobound}[1]{k({#1})}%
    By~\cref{cl:TreewidthBound} and~\cref{cl:MaxDegreeBound}, we can now apply~\cref{lem:BoundedAdhesionAndTorsosViaTreewidthAndDegree} to~$G'$ and obtain a {\tcd}~$(T,\cX')$ of~$G'$ of adhesion at most~$\adhesionbound{\ell} := (2 \treewidthbound{\ell} + 2) \maxdegreebound{\ell}$ such that every torso of~$(T,\cX')$ has size at most~$\torsobound{\ell} := (\maxdegreebound{\ell} + 1)(\treewidthbound{\ell} + 1)$.
    From the {\tcd}~$(T, \cX')$ of the minor~$G'$ of~$G$, we then construct the {\tcd}~$(T, \cX)$ of~$G$ by defining, for each node~$t \in T$, its corresponding part~$X_t$ to be the union of the branch sets of the vertices in~$X_t'$.
    
    Now let
    \begin{equation*}
        \alpha = \alpha(\ell) := \max\{\torsobound{\ell}( 8 \ell^2 + 1 + 6 \ell^2 \cdot 16 \ell^4), p(\ell) \cdot (6 \ell^2)^2 + \binom{16 \ell^4}{2} \cdot \paralleledges{\ell} \cdot (6 \ell^2)^2 + \torsobound{\ell} \cdot \maxdegreebound{\ell}/2 + \adhesionbound{\ell} \cdot \torsobound{\ell}\}.
    \end{equation*}
    \noindent We show that~$(T, \cX)$ has the desired properties in that it has adhesion at most~$\alpha$ and the~$3$-centres of its torsos are~\almostalphathin.
    In doing so, the first term in the maximum defining~$\alpha$ will be used to bound the size of the set of vertices that we will delete from the~$3$-centre of a torso to make it~\alphathin\ as well as the number of its neighbours, while the second term in the maximum will be used to show that the remainder of the~$3$-centre is \alphathin.
    
    \begin{claim} \label{cl:AdhesionBound}
       The {\tcd}~$(T, \cX)$ has adhesion at most~$\adhesionbound{\ell} \le \alpha$.
    \end{claim}
    \begin{claimproof}
        By the construction of~$(T, \cX)$ from~$(T, \cX')$, the adhesion sets of~$(T, \cX)$ are the same as those of~$(T, \cX')$.
        Thus, $(T, \cX)$ has the same adhesion as~$(T, \cX')$ and, in particular, adhesion at most~$\adhesionbound{\ell}$.
    \end{claimproof}

    \begin{claim} \label{cl:AlmostAlphaThin}
        The~$3$-centres of the torsos of~$(T, \cX)$ are~\almostalphathin.
    \end{claim}
    \begin{claimproof}
        We first observe that the torsos of~$(T, \cX)$ equal their~$3$-centres:
        Every adhesion set of~$(T, \cX)$ has size at least~$3$, since~$G$ is~$3$-edge-connected.
        Therefore, all peripheral vertices of torsos of~$(T, \cX)$ have degree at least~$3$ in the respective torso.
        This implies that the~$3$-centre of a torso of~$(T, \cX)$ equals the torso itself.%
        \COMMENT{}

        Given any node~$t \in T$, let~$H_t$ be the torso of~$(T, \cX)$ at~$t$.
        We first find a set~$X$ of vertices of~$H_t$ such that~$A[X_t] - X$ is a disjoint union of paths, and we show that both the size of~$X$ and the size of its neighbourhood in~$H_t$ are bounded by our suitably chosen~$\alpha = \alpha(\ell)$.
        In a next step, we then use the structure of~$A[X_t] - X$ to exhibit a natural enumeration of~$V(H_t - X)$, and we finally show that this enumeration witnesses the~\alphathin ness of~$H_t - X$.
        
        For every component~$C$ of~$A[X_t]$, and more generally for every component~$C$ of~$A$, the graph~$A[C]$ does not contain the star~$K_{1, 2 \ell^2}$ as a minor.
        To prove this, we show that~$G[C]$ otherwise contains~$W_\ell$ as a strong immersion minor, contradicting our assumptions on~$G$.
        By~\cref{lem:SpiderContainsAll}, it is enough to show that~$S_{3,2 \ell^2}$ is strongly immersed in~$G[C]$.
        So let~$U \subseteq C$ be a set of vertices containing precisely one vertex of each branch set of the~$K_{1, 2\ell^2}$-minor of~$A[C]$.
        We then map the centre of~$S_{3,2 \ell^2}$ to the vertex of~$U$ in the branch set of the centre of~$K_{1,2\ell^2}$, and map the rest of the vertices of~$S_{3,2 \ell^2}$ bijectively to the remaining ones in~$U$.
        Since~$\paralleledges{\ell} \geq 6 \ell^2$, we can now greedily embed the edges of~$S_{3,2 \ell^2}$ as edge-disjoint paths in~$G[C]$ joining the vertices in~$U$ along the~$K_{1,2\ell^2}$-minor.
        Hence, we have found our desired strong immersion of~$S_{3,2 \ell^2}$ in~$G[C]$.%
        \COMMENT{}
        
        We are now ready to find the desired set~$X$ of vertices of the torso~$H_t$ of~$(T, \cX)$ at~$t$.
        For every component~$C$ of~$A[X_t]$, the graph~$A[C]$ does not contain~$K_{1, 2\ell^2}$ as a minor. So~\cref{lem:ExcludedStar} yields a set $X_C \subseteq C$ of size at most~$8 \ell^2$ such that~$A[C] - X_C$ is a disjoint union of paths.
        We then let the set~$X$ be the union of these~$X_C$.
        The construction of~$X$ then implies that~$A[X_t] - X$ is indeed a disjoint union of paths.

        We first show that the size of~$X$ is bounded by~$\alpha$.
        Since~$A$ has maximum degree at most~$2\ell^2$ (as in the proof of~\cref{cl:MaxDegreeBound}) and $A[C]$ is connected, the number of components of~$A[C] - X_C$ is at most~$|X_C| \cdot 2 \ell^2 \leq 16 \ell^4$.
        Since~$X_t'$ has size at most~$\torsobound{\ell}$, the graph~$A[X_t]$ has at most~$\torsobound{\ell}$ components and thus~$|X| \le 8 \ell^2 \cdot \torsobound{\ell}  \le \alpha$.
        
        Next, we claim that each vertex~$x$ in~$X$ has at most~$8 \ell^2 \cdot \torsobound{\ell} + \torsobound{\ell} + \torsobound{\ell} \cdot 6 \ell^2 \cdot 16 \ell^4 \leq \alpha$ neighbours in~$H_t$.
        Indeed, suppose for a contradiction that~$x$ has more than~$\torsobound{\ell} + \torsobound{\ell} \cdot 6 \ell^2 \cdot 16 \ell^4$ neighbours in~$H_t - X$.%
        \COMMENT{}
        We first note that~$x$ is adjacent to at most~$\torsobound{\ell}$ peripheral vertices of~$H_t$, since there are at most~$\torsobound{\ell}$ many peripheral vertices of~$H_t$: They are in one-to-one correspondence to the peripheral vertices of the torso of~$(T, \cX')$ at~$t$ which has size at most~$\torsobound{\ell}$.
        Thus, $x$ has more than~$\torsobound{\ell} \cdot 6 \ell^2 \cdot 16 \ell^4$ many neighbours which are core vertices of~$H_t$ and not contained in~$X$.
        By the pigeonhole principle, there then exists one component~$C$ of the at most~$\torsobound{\ell}$ components of~$A[X_t]$ such that~$x$ has at least~$6 \ell^2 \cdot 16 \ell^4$ neighbours in~$C \setminus X_C$.%
        \COMMENT{}
        Applying the pigeonhole principle a second time, we find one component~$P$ of the at most~$16 \ell^4$ components of~$A[C] - X_C$ that contains at least~$6 \ell^2$ neighbours of~$x$.%
        \COMMENT{}
        Since~$A[P]$ is a path by the choice of~$X_C$, we have hence found a subdivision and thus a strong immersion of~$P_{6 \ell^2-1} * v$ in~$H_t$.
        As all the vertices in~$C$ as well as~$x$ are core vertices of the torso~$H_t$, this also is a strong immersion of~$P_{6 \ell^2-1} * v$ in~$G$.
        So by~\cref{lem:PathWithApexContainsWall}, the wall~$W_\ell$ is contained in~$G$ as a strong immersion minor -- a contradiction.

        Since~$X$ has size at most~$\alpha$ and each element of~$X$ has at most~$\alpha$ neighbours in~$H_t$, it remains to prove that~$H_t - X$ is~\alphathin.
        To do so, we now give an enumeration of~$V(H_t - X)$ which we will then show to witness the~\alphathin ness of~$H_t - X$.
        
        Let~$c_1, \dots, c_n$ be an arbitrary enumeration of the part~$X_t'$ of the \tcd~$(T, \cX')$ of~$G'$, let~$C_i$ be the component of~$A$ corresponding to the vertex~$c_i$ of~$G'$, and write~$X_i := X_{C_i}$.
        Then take an enumeration~$v_1, \dots, v_N$ of~$X_t \setminus X$ such that for~$i = 1, \dots, n$, the set~$C_i - X_i$ forms an interval of this enumeration and within such an interval, every component of~$A[C_i] - X_i$ forms an interval with the same linear order which is given by the path that it induces in~$A$ by the choice of~$X_i$.
        For the peripheral vertices of~$H_t$, we finally pick an arbitrary enumeration~$v_{N+1}, \dots, v_M$.

        To show that this enumeration~$v_1, \dots, v_M$ of~$V(H_t-X)$ is as desired, first pick an arbitrary~$i \in \{1, \dots, N\}$.
        Then there is a component~$C_j$ of~$A[X_t]$ with~$v_i \in C_j$, and a component of~$A[C_j] - X_j$ containing~$v_i$; this component of~$A[C_j] - X_j$ then induces a path~$P$ in~$A$ by the choice of~$X_j$ (which might have length~$0$).
        Every edge in~$H_t$ joining~$\{v_1, \dots, v_{i-1}\}$ and~$\{v_{i+1}, \dots, v_M\}$ belongs to one of the following classes:
        \begin{enumerate}[label*=(\arabic*)]
        	\item edges joining vertices of the path~$P$, \label{item:EdgesInsidePath}
        	\item edges joining the distinct components in~$A[C_j] - X_j$, \label{item:EdgesBetweenPaths}
        	\item edges joining distinct components of~$A[X_t']$, and \label{item:EdgesBetweenComponents}
        	\item edges incident to peripheral vertices of~$H_t$. \label{item:EdgesToPeripheralVertices}
      	\end{enumerate}
        We will now separately bound the size of each of these classes~\labelcref{item:EdgesInsidePath,item:EdgesBetweenPaths,item:EdgesBetweenComponents,item:EdgesToPeripheralVertices} in terms of~$\ell$.
        Note that only~\labelcref{item:EdgesInsidePath} depends on the choice of the enumeration; in particular, we will indeed bound the absolute number of all such edges in~$H_t$ in terms of~$\ell$ for~\labelcref{item:EdgesBetweenPaths,item:EdgesBetweenComponents,item:EdgesToPeripheralVertices}.

        For~\labelcref{item:EdgesInsidePath}, the edges joining vertices of~$P$, note that~$V(P)$ is a component of~$A[C_j] - X_j$.
        So there are no edges in~$A$ joining the vertices of~$P$ except the ones contained in~$P$ itself.
        Hence, each two non-adjacent vertices of~$P$ are joined by at most~$\paralleledges{\ell}$ parallel edges. 
        Let~$P_1$ and~$P_2$ be the two components of~$P - v_i$, and note that both~$P_1$ and~$P_2$ consist of vertices of~$G$ since all vertices on~$P$ are in~$X_t$.
        If~$|E_G(P_1, P_2)| > p(\ell) \cdot (6 \ell^2)^2$, then one of~$P_i$, say~$P_1$, has more than~$6 \ell^2$ neighbours in~$P_{2 - i}$.
        Fix any~$v \in P_1$.
        Since~$p(\ell) \ge 6\ell^2$, we can greedily find edge-disjoint paths in~$G[P_1 \cup P_2]$ joining~$v$ and~$6 \ell^2$ of these neighbours in~$P_2$, and we thus find a strong immersion of~$P_{6 \ell^2 - 1} * v$ in~$G[P_1 \cup P_2]$.
        By~\cref{lem:PathWithApexContainsWall}, this graph contains the wall~$W_{\ell}$ as a strong immersion minor which yields the desired contradiction.
        
        With the same type of argument, one can show that there are at most~$p(\ell) \cdot (6 \ell^2)^2$ edges in~$G$ joining any two components of~$A[C_j] - X_j$.
        Recall that the number of components of~$A[C_j] - X_j$ is at most~$16 \ell^4$.
        Thus, there are at most~$\binom{16 \ell^4}{2} \cdot \paralleledges{\ell} \cdot (6 \ell^2)^2$ edges in~$G$ as in~\labelcref{item:EdgesBetweenPaths}, i.e.\ edges joining components of~$A[C_j] - X_j$.

        Next, we turn to~\labelcref{item:EdgesBetweenComponents} and bound the number of edges in~$G$ joining distinct components of~$A[X_t']$.
        Since~$G'$ arises from~$G$ by contracting the components of~$A$, these edges are in one-to-one correspondence with the edges in~$G'[X_t']$.
        Since~$X_t'$ has size at most~$\torsobound{\ell}$ and~$G'$ has maximum degree at most~$\maxdegreebound{\ell}$, there are at most~$\torsobound{\ell} \cdot \maxdegreebound{\ell} / 2$ edges in~$G'[X_t']$ and hence joining distinct components of~$A[X_t']$.

        Along similar lines, we can bound the number of edges in~\labelcref{item:EdgesToPeripheralVertices}, i.e.\ the edges incident to peripheral vertices of~$H_t$, in terms of~$\ell$.
        Indeed, the edges incident to peripheral vertices of~$H_t$ are in one-to-one correspondence to the edges incident to the peripheral vertices of the torso of~$(T, \cX')$ at~$t$.
        Since~$(T, \cX')$ has adhesion at most~$\adhesionbound{\ell}$, each peripheral vertex has degree at most~$\adhesionbound{\ell}$.
        Moreover, since~$X_t'$ has size at most~$\torsobound{\ell}$, there are at most~$\torsobound{\ell}$ peripheral vertices of~$H_t$.
        Altogether, there are at most~$\adhesionbound{\ell} \cdot \torsobound{\ell}$ edges incident to peripheral vertices of~$H_t$.

        All the above bounds show that by the choice of~$\alpha = \alpha(\ell)$, there are at most~$\alpha$ edges joining~$\{v_1, \dots, v_{i-1}\}$ and~$\{v_{i+1}, \dots, v_M\}$ in~$H_t$ for~$i = 1, \dots, N$.
        By the choice of our enumeration, the vertices~$v_{N+1}, \dots, v_M$ are precisely the peripheral vertices of~$H_t$.
        So for~$i = N+1, \dots, M$, all edges joining~$\{v_1, \dots, v_{i-1}\}$ and~$\{v_{i+1}, \dots, v_M\}$ in~$H_t$ are incident to peripheral vertices of~$H_t$.
        As shown above there are at most~$\adhesionbound{\ell} \cdot \torsobound{\ell} \le \alpha(\ell)$ such edges.
        This completes the proof of this claim.
    \end{claimproof}
    
    \noindent By~\cref{cl:AdhesionBound} and~\cref{cl:AlmostAlphaThin}, our constructed {\tcd}~$(T, \cX)$ of~$G$ now is as desired, completing the proof of~\cref{main:NoWallImpliesAlphaThin}.
\end{proof}%
\COMMENT{}

\section*{Acknowledgement}
    
We wrote this post-publication version of the paper after Emanuele Natale and \'Edouard Oyallon alerted us to the fact that their AI~project {\em Mathpocalypse\/}~\cite{Mathpocalypse2026} had flagged Lemma~2.4 in our published paper as incorrect, which resulted in a gap later in the main proof. We then tasked our own~AI%
   \COMMENT{}
   with fixing the~proof. It came up with two independent and quite different fixes, and also flagged various misprints.

This version of the paper is a manual revision of the original that follows the idea of one of those two fixes. The material changes are limited to \cref{subsec:Reduction3Connected}; in particular, Lemma~2.4 is no longer needed. But we have included a corrected version as \cref{lem:AlphaThinPrimeSubdividingSuppressing}, which is interesting to note and very short.

\bibliographystyle{amsplain}
\bibliography{collective.bib}

\end{document}

%% file: Preamble.tex
\usepackage{amsmath}
\usepackage{amssymb}
\usepackage{amsthm}
\usepackage{mathtools}
\usepackage{enumitem}
\setenumerate{label={\normalfont (\roman*)}}

\usepackage[utf8]{inputenc}
\usepackage[T1]{fontenc}
\usepackage{lmodern}
\usepackage[babel]{microtype}
\usepackage[english]{babel}
\usepackage{relsize}

\usepackage{graphicx}
\usepackage{subcaption}

\usepackage{geometry}
\usepackage{xcolor} 	
\usepackage{hyperref}
\hypersetup{
	colorlinks,
	linkcolor={red!60!black},
	citecolor={green!60!black},
	urlcolor={blue!60!black},
}
\usepackage[abbrev, msc-links]{amsrefs}
\usepackage[nameinlink, capitalise, noabbrev]{cleveref}
\crefformat{enumi}{#2#1#3}
\crefformat{equation}{#2(#1)#3}
\crefname{mainresult}{Theorem}{Theorems}
\usepackage{doi}

\renewcommand{\PrintDOI}[1]{\doi{#1}}

\let\setminus=\smallsetminus

\renewcommand{\leq}{\leqslant}
\renewcommand{\geq}{\geqslant}
\renewcommand{\ge}{\geq}
\renewcommand{\le}{\leq}

\newtheorem{theorem}{Theorem}[section] 
\newtheorem{proposition}[theorem]{Proposition}

\newtheorem{lemma}[theorem]{Lemma}

\newtheorem{example2}[theorem]{Example}
\newtheorem{mainresult}{Theorem} 

\newtheorem{innercustomthm}{}

\newenvironment{customthm}[1]
  {\renewcommand\theinnercustomthm{#1}\innercustomthm}
  {\endinnercustomthm}

\theoremstyle{definition}

\newtheorem{definition}[theorem]{Definition}

\theoremstyle{remark}

\newtheorem{claim}{Claim}
\crefname{claim}{Claim}{Claims}

\newenvironment{claimproof}{\noindent\textit{Proof.}}{\hfill\ensuremath{\blacksquare}\medskip}
\usepackage{etoolbox}

\newcommand{\COMMENT}[1]{{}}

\let\theta=\vartheta
\let\rho=\varrho
\let\phi=\varphi

\makeatletter

\def\calCommandfactory#1{%
  \expandafter\def\csname c#1\endcsname{\mathcal{#1}}}
\def\frakCommandfactory#1{%
  \expandafter\def\csname frak#1\endcsname{\mathfrak{#1}}}
   
\newcounter{ctr}
\loop
  \stepcounter{ctr}
  \edef\X{\@Alph\c@ctr}
  \expandafter\calCommandfactory\X
  \expandafter\frakCommandfactory\X
\ifnum\thectr<26
\repeat





%% file: RevisionAfterAIforArXiv.bbl
\begin{bibdiv}
\begin{biblist}

\bib{abu2003graph}{inproceedings}{
      author={Abu-Khzam, Faisal~N.},
      author={Langston, Michael~A.},
       title={Graph coloring and the immersion order},
organization={Springer},
        date={2003},
   booktitle={International computing and combinatorics conference},
       pages={394\ndash 403},
}

\bib{burger2022duality}{article}{
      author={B{\"u}rger, Carl},
      author={Kurkofka, Jan},
       title={Duality theorems for stars and combs {IV}: Undominating stars},
        date={2022},
     journal={Journal of Graph Theory},
      volume={100},
      number={1},
       pages={140\ndash 162},
}

\bib{devos2014minimum}{article}{
      author={DeVos, Matt},
      author={Dvo{\v{r}}{\'a}k, Zden{\v{e}}k},
      author={Fox, Jacob},
      author={McDonald, Jessica},
      author={Mohar, Bojan},
      author={Scheide, Diego},
       title={A minimum degree condition forcing complete graph immersion},
        date={2014},
     journal={Combinatorica},
      volume={34},
      number={3},
       pages={279\ndash 298},
}

\bib{devos2013note}{article}{
      author={DeVos, Matt},
      author={McDonald, Jessica},
      author={Mohar, Bojan},
      author={Scheide, Diego},
       title={A note on forbidding clique immersions},
        date={2013},
     journal={The Electronic Journal of Combinatorics},
      volume={20},
      number={3},
        note={{\#}P55},
}

\bib{DiestelBook24}{book}{
      author={Diestel, R.},
       title={Graph theory \emph{(6th edition)}},
   publisher={Springer-Verlag},
        date={2025},
        note={\\ Electronic edition available at {\small\tt http://diestel-graph-theory.com/}},
}

\bib{dvorak2014strong}{article}{
      author={Dvo{\v{r}}{\'a}k, Zdenek},
      author={Klimosova, Tereza},
       title={Strong immersions and maximum degree},
        date={2014},
     journal={SIAM Journal on Discrete Mathematics},
      volume={28},
      number={1},
       pages={177\ndash 187},
}

\bib{dvorak2016structure}{article}{
      author={Dvo{\v{r}}{\'a}k, Zden{\v{e}}k},
      author={Wollan, Paul},
       title={A structure theorem for strong immersions},
        date={2016},
     journal={Journal of Graph Theory},
      volume={83},
      number={2},
       pages={152\ndash 163},
}

\bib{ferrara2008h}{article}{
      author={Ferrara, Michael},
      author={Gould, Ronald~J.},
      author={Tansey, Gerard},
      author={Whalen, Thor},
       title={On {H}-immersions},
        date={2008},
     journal={Journal of Graph Theory},
      volume={57},
      number={3},
       pages={245\ndash 254},
}

\bib{VxMinorsGrid}{article}{
      author={Geelen, J.},
      author={Kwon, O-J.},
      author={McCarty, R.},
      author={Wollan, P.},
       title={The grid theorem for vertex minors},
        date={2023; arXiv:1909.08113},
     journal={J. Comb. Theory Ser. B},
      volume={158},
       pages={93\ndash 116},
}

\bib{GKMW-TopSubgraphsFPT}{article}{
      author={Grohe, M.},
      author={Kawarabayashi, K.},
      author={Marx, D.},
      author={Wollan, P.},
       title={Finding topological subgraphs is fixed parameter tractable},
        date={2011},
     journal={Proc. ACM Symp. on Theory of Computing (STOC)},
       pages={479\ndash 488},
}

\bib{harvey2017parameters}{article}{
      author={Harvey, Daniel~J.},
      author={Wood, David~R.},
       title={Parameters tied to treewidth},
        date={2017},
     journal={Journal of Graph Theory},
      volume={84},
      number={4},
       pages={364\ndash 385},
}

\bib{DigraphGridThm}{inproceedings}{
      author={Kawarabayashi, Ken-ichi},
      author={Kreutzer, Stephan},
       title={The directed grid theorem},
        date={2015},
   booktitle={{Proceedings of the forty-seventh annual ACM symposium on Theory of Computing}},
       pages={655\ndash 664},
}

\bib{Immersionminimalinfinitelyedgecongraph}{article}{
      author={Knappe, Paul},
      author={Kurkofka, Jan},
       title={The immersion-minimal infinitely edge-connected graph},
        date={2024},
     journal={J.\~Comb.\ Theory Ser.~B},
      volume={164},
       pages={492\ndash 516},
}

\bib{marx2014immersions}{article}{
      author={Marx, D{\'a}niel},
      author={Wollan, Paul},
       title={Immersions in highly edge connected graphs},
        date={2014},
     journal={SIAM Journal on Discrete Mathematics},
      volume={28},
      number={1},
       pages={503\ndash 520},
}

\bib{Mathpocalypse2026}{article}{
      author={Natale, Emanuele},
      author={Oyallon, {\'E}douard},
       title={{Mathpocalypse}},
        date={2026},
     journal={{https://github.com/natema/mathpocalypse-project; DOI: 10.5281/zenodo.21499917}},
         url={https://github.com/natema/mathpocalypse-project},
}

\bib{GM}{article}{
      author={Robertson, N.},
      author={Seymour, P.D.},
       title={Graph minors {I--XX}},
        date={1983--2004},
     journal={J. Comb. Theory Ser. B},
}

\bib{GMV}{article}{
      author={Robertson, N.},
      author={Seymour, P.D.},
       title={Graph minors. {V}. {E}xcluding a planar graph},
        date={1986},
     journal={J.~Comb.\ Theory, Series~B},
      volume={41},
       pages={92\ndash 114},
}

\bib{GMXIII}{article}{
      author={Robertson, N.},
      author={Seymour, P.D.},
       title={Graph minors. {XIII}. {T}he disjoint paths problem},
        date={1995},
     journal={J.~Comb.\ Theory, Series~B},
      volume={63},
       pages={65\ndash 110},
}

\bib{GMXXIII}{article}{
      author={Robertson, N.},
      author={Seymour, P.D.},
       title={{Graph minors. XXIII. Nash-Williams' immersion conjecture}},
        date={2010},
     journal={J.~Comb.\ Theory, Series~B},
      volume={100},
       pages={181\ndash 205},
}

\bib{WollanImmersion}{article}{
      author={Wollan, Paul},
       title={The structure of graphs not admitting a fixed immersion},
        date={2015},
     journal={J. Comb. Theory Ser. B},
      volume={110},
       pages={47\ndash 66},
}

\end{biblist}
\end{bibdiv}
